\documentclass{imsart}
\RequirePackage{natbib}

\usepackage{graphicx}
\usepackage{amsmath, amstext, amssymb,mathbbol, amsthm, multirow}
\RequirePackage[colorlinks,citecolor=blue,urlcolor=blue]{hyperref}
\usepackage{color}

\newtheorem{theorem}{Theorem}[section]

\newtheorem{lemma}[theorem]{Lemma}

\newtheorem{remark}[theorem]{Remark}
\newtheorem{example}{Example}[section]
\newtheorem{definition}{Definition}[section]

\providecommand{\abs}[1]{\lvert#1\rvert}

\def\<<{``}

\def\P{\mathbb{P}}

\usepackage{color}

\begin{document}

\begin{frontmatter}
	
	\title{Strong approximation of density dependent Markov chains on bounded domains}
	\runtitle{Strong approximation of DDMC on bounded domains}
	
	
	\author{\fnms{Enrico} \snm{Bibbona}\ead[label=e1]{enrico.bibbona@polito.it}}
	\address{DISMA, Politecnico di Torino\\
		Corso Duca degli Abruzzi, 24 \\
		10129 Torino (TO), Italy\\		
		\printead{e1}}
	\affiliation{Politecnico di Torino, Italy}
	\and
	\author{\fnms{Roberta} \snm{Sirovich}\ead[label=e2]{roberta.sirovich@unito.it}}
	\address{Dipartimento di Matematica ``G. Peano''\\
		Universit\`a di Torino\\
		via Carlo Alberto, 10\\
	 10123 Torino, Italy\\
	 \printead{e2}}
	\affiliation{Universit\`a di Torino, Italy}
	
	\runauthor{E. Bibbona, R. Sirovich}
	
	\begin{abstract}
	Density dependent families of Markov chains, such as the stochastic models of mass-action chemical kinetics, converge for large values of the indexing parameter $N$ to deterministic systems of differential equations \citep{kurtz1ode}. Moreover for moderate $N$ they can be strongly approximated by paths of a diffusion process \citep{kurtz1976limit}. Such an approximation however fails if the state space is bounded (at zero or at a constant maximum level due to conservation of mass) and if the process visits the boundaries with non negligible probability. We present a strong approximation by a jump-diffusion process which is robust to this event. The result is illustrated with a particularly hard case study.
	\end{abstract}
	

	
\end{frontmatter}

\section{Introduction}

Density dependent Markov chains are widely used models in many fields. Their origin is in population dynamics, but nowadays the most popular application is to chemical reaction networks. In systems biology indeed networks of reactions that take place within a single cell are of high interest and in such a small scale, stochastic effects are commonly observed and the deterministic models that are largely predominant for chemical systems of larger size may fail to catch relevant properties. The coexistence of stochastic and deterministic models of mass-action chemical kinetics is at the basis of some works by Tom Kurtz \cite{kurtz1ode, KurtzChemPhys}, who first clarified the relation among them. Stochastic models are indeed defined as families of Markov chains that are indexed by a parameter usually interpreted as the volume of reaction. For large values of such parameter they converge to their deterministic counterparts. The convergence holds on an arbitrarily large, but finite, time interval $[0,T]$. Despite this fact, it is still possible to find numerous examples where the behavior of the stochastic models is substantially different with respect to the deterministic one. Many such examples are discussed in the recent reviews \cite{ErdiLente,andersonKurtzBook}.
It is particularly worth to notice that the asymptotic distribution of the stochastic models is not always guaranteed to concentrate around the deterministic equilibrium. It happens in some important examples (cf. \cite{andersonCKstationary,danielecarsten}), but there are entire families of reaction networks where the long-term behavior is dramatically different if modeled deterministically or with stochastic processes \cite{andersonencisojohnston}.

Another important class of examples where stochastic and deterministic models may differ substantially is when the state space of the system is bounded and boundaries are visited with non-negligible probability. It is the case of some reaction networks, indeed the concentrations cannot become negative and in some case they also cannot grow arbitrarily large by a conservation of mass constraint.
 It has been shown in many networks (cf. e.g. \cite{srivastava, TogashiPRL,ATPN14, angius2015approximate}), that the state prescribed by the deterministic equations at a given time may not correspond to a state around which the probability mass is concentrated in the corresponding stochastic model even if the indexing parameter is moderately large.
 Clearly, a purely deterministic model cannot catch effects that are only caused by the presence of noise, such as the reaching of the boundaries of the state space when the mean is in the interior. Thus different approximations that retain some stochasticity could certainly be more accurate. Kurtz himself, cf. \cite{kurtz1976limit,kurtzbook}, showed that such density dependent Markov chain models can be strongly approximated with path of diffusion processes. Again, however, if state space is bounded and if the boundaries are frequently visited, such an approach fails:  the diffusion approximation is indeed valid only up to the first visit to the boundary and then it becomes ill-posed.
 To overcome this problem at least two solutions have been proposed.
 In \cite{ATPN14, angius2015approximate} a jump diffusion approximation was investigated in a few case studies. It was shown, by simulations in most cases and by numerical solution of the Kolmogorov equations in a one dimensional case, that both a good agreement with the original Markov chain model and a significant computational advantage may be achieved. 
 However, no mathematical study of the properties of  the approximation was attempted there, and only some numerical evidence in specific networks was provided.
 In \cite{ruth2017constrained} a similar but different sequence of approximating jump diffusion processes  was proved to converge to a reflected diffusion process. The result was also illustrated by some simulation study and in the last part of this paper we replicate one of their examples.
 A further method  was introduced in \cite{complex} in some special case by extending the domain of the diffusion approximation to the complex plain.
 In this paper we aim at showing that the idea on which the jump diffusion approximation first proposed in \cite{ATPN14} can be actually formalized and turned into a strong approximation theorem which provides us with the same rate of approximation that was found by Kurtz in the case of diffusion approximation. However the newly constructed jump-diffusion process is able to approximate the original Markov chain correctly even after many visits to the boundary, while the diffusion approximation cannot be extended after the first hitting time to the boundary. The result is illustrated by a couple of a few case studies which are particularly complex from the numerical point of view.
\section{Background material} \label{sec:background}

In the Introduction we have mentioned density dependent families of Markov chains as models for population dynamics. Besides the range of applicability of such processes is much wider, it is convenient to keep such example in mind since the nomenclature originated in such setting. We have also mentioned that when populations are spread over large areas, their densities are well approximated by ODE systems.

 To find the precise relation to stochastic population models and their deterministic ODE counterparts a more general and formal definition of density dependence is needed. In this section we summarize some classic results due to Tom Kurtz \citep{kurtz1ode, kurtz1976limit, kurtz1978strong}. Except for some specific issue that we comment below, we follow the systematic presentation in the book \cite{kurtz}[Chapter 11], where some of the original results have been extended and generalized and where the reader can find every detail that we are going to skip.
 Let 
	\begin{definition}
	A family of continuous time Markov Chains $\{Y^{[n]}(t)\}$ indexed by a parameter $n$ and with state spaces contained in $\mathbb{Z}^d$ is called \emph{density dependent} if its transition rates $q^{[n]}_{k,k+l}$ from any state $k$ to any other state $k+l$ can be written in the following form
	\begin{equation}\label{rates}
	q^{[n]}_{k,k+l}=n\; f_l \left(\frac{k}{n}\right)
	\end{equation}
	with $f_l$ non-negative functions defined on some subset of $\mathbb{R}^d$.
	\end{definition}

In the case population dynamics, the process $Y^{[n]}(t)$ usually represents the population size and the indexing parameter $n$, the area over which the population is spread. It is common practice to rescale the populations into population densities, and the general analogue of such procedure is introduced below.

	\begin{definition}
		For a density dependent family $\{Y^{[n]}(t)\}$ we define the family of \emph{density processes} $\{X^{[n]}(t)\}$ by setting for every $n$
		\begin{equation} X^{[n]}(t)=\frac{Y^{[n]}(t)}{n}. \label{density} \end{equation}
	\end{definition}

Each density process of a density dependent family can be written with the following representation
\begin{align}
X^{[n]}(t) &= X^{[n]}(0)+\sum_{l} \frac{l}{n} N_{l}\left[ n \int_{0}^{t} f_{l}\left(X^{[n]}(s)\right) ds\right]\notag \\
&=X^{[n]}(0)+ \sum_{l} \frac{l}{n} \left\{n \int_0^t f_l\left(X^{[n]}(s)\right) ds+ \tilde N_{l}\left[ n \int_{0}^{t} f_{l}\left(X^{[n]}(s)\right) ds\right]\right\}.\label{mc}
\end{align}
where the $N_l(t)$ are independent unit-rate Poisson processes that count the occurrences of the events whose effect is to increment $Y^{[n]}(t)$ by $l$ (or the density process $X^{[n]}(t)$ by $\frac{l}{n}$). By $\tilde N_{l}(t)$ we denote the correspondent compensated processes $\tilde N_{l}(t)=N_{l}(t)-t$.
	
When $n$ gets large the jumps of the stochastic process \eqref{mc} become both more frequent and smaller in magnitude suggesting that the cadlag trajectories of \eqref{mc} could be approximated by continuous functions (the so-called \emph{fluid} limit or fluid approximations). The process indeed converges to the deterministic solution $x(t)$ of the following $d$-dimensional ODE system
\begin{equation}\label{ode}
\dot x=F(x(t))=\sum_{l\in C} l f_l(x(t))
\end{equation}
under rather general assumptions.
Let us remark that the state space of $X^{[n]}$ is countable while the deterministic process \eqref{ode} lives in a subset of $\mathbb{R}^d$. We set the following notations. Let $E_i^{[n]}$ be the smallest interval (either closed or unbounded) that contains all possible values of the $i$-th component of $X^{[n]}$. Moreover, we denote by $E^{[n]}=\times_i E_i^{[n]}$ the hyperrectangle where the vector $X^{[n]}$ takes values and $\displaystyle E=\cup_{n=0}^\infty E^{[n]}$. In the easiest and most common case, the total population size is conserved in time, moreover it is common practice to set the same initial condition $X_n(0)$ for every $n$ so that they determine the same endpoints of the intervals $E_i^{[n]}$ for all the $n$, while the countable lattice which forms the state space of $X^{[n]}$ becomes denser and denser in $E$ as $n\rightarrow\infty$. The hyperrectangle $E$ can be taken as the state space of the deterministic approximation \eqref{ode}.

A precise statement of this convergence result follows.
\begin{theorem}
Suppose that for each compact $K \subset E$,
\[ \sum_l \abs{l} \sup_{x\in K} f_l (x) < \infty \]
and the function $F$ defined in \eqref{ode} is Lipshitz continuous in $K$, suppose that for each $n$  $X^{[n]}(t)$ satisfies \eqref{mc} with initial conditions $X^{[n]}(0)$ such that $\lim_{n\rightarrow \infty}X^{[n]}(0) =x_0$  and $x(t)$ solves \eqref{ode} with initial condition $x(0)=x_0$, then for every $t\geq 0$,
\[ \lim_{n\rightarrow \infty} \sup_{s\leq t}\,\abs{X^{[n]}(s)-x(s)}=0\qquad \text{a.s.}\]
\end{theorem}

Such a deterministic limit allows to approximate the mean of the density process \eqref{mc}, but looses all other informations encoded in its randomness like variance, skewness, bimodalities, tail behaviors\ldots
 
A sharper continuous approximation can be obtained in terms a diffusion process as stated in the following Theorem. Our formulation here is still very close to the one in \cite{kurtz} except that we want to give special emphasis to some assumptions that are only implicitly made in the book, but that are crucial for the motivations of present article (cf.Remark \ref{remBound}). More comments on the assumptions are included at the beginning of Section \ref{robabella}.

\begin{theorem}\label{diffapprox}
Let $X^{[n]}(t)$  be as in \eqref{mc} and let $x(t)$ solve \eqref{ode} with initial condition $x(0)=x_0$. Let moreover the diffusion process $G^{[n]}(t)$ solve
\begin{equation}
G^{[n]}(t) = G^{[n]}(0)
+\sum_{l} \frac{l}{n}\, \left[n\int_{0}^{t} f_{l}(G^{[n]}(s))ds+ W_{l}\left(n \int_{0}^{t} f_{l}(G^{[n]}(s))ds\right)\right]\label{diff}
\end{equation}
where the $W_l(t)$ are independent standard Wiener processes. For every $l$ the driving processes $N_l(t)$ and $W_l(t)$ of equations \eqref{mc} and \eqref{diff} be constructed on the same probability space as in Lemma \ref{lemma:MBePP}.
Let $U$ be any open connected subset of $E^{[n]}$ that contains $x(t)$ for every $0\leq t\leq T$. Let $\bar f_l=\sup_{x \in U} f_l(x)<\infty$ and suppose $\bar f_l=0$ except for finitely many $l$. Suppose $M>0$ satisfies both the two equations below for any $x,y\in U$
\begin{equation}\begin{aligned}
\abs{f_l(x)-f_l(y)}&\leq M\abs{x-y}\\
\abs{F(x)-F(y)}&\leq M\abs{x-y}.
\end{aligned}\label{Lip}\end{equation}
Let $\tau_n=\inf\{t:X^{[n]}(t)\notin U \textup{ or }G^{[n]}(t)\notin U\}$. Note that $\mathbb{P}(\tau_n>T)\rightarrow 1$ for $n\rightarrow\infty$. Then for $n\geq 2$ there is a random variable $\Gamma^T_n$ and positive constants $\lambda_T$, $C_T$, and $K_T$ depending on $T$, on $M$, and on the $\bar f_l$, but not on $n$, such that
\[\sup_{t\leq T\wedge \tau_n} \Big| X^{[n]}(s)-G^{[n]}(s) \Big| \leq \Gamma^T_n \frac{\log n}{n}\]
and
\[\mathbb{P}\left(\Gamma^T_n>C_T+x\right)\leq K_T\, n^{-2}\exp\left(-\lambda_T\sqrt{x}-\frac{\lambda_T x}{\log n}\right).\]
\end{theorem}

Equation \eqref{mc} and equation \eqref{diff} have the same structure, the sole difference being that Brownian motions replace compensated Poisson processes. 

One important ingredient in the proof of Theorem \ref{diffapprox} is indeed that paired trajectories of a Wiener process and of a Poisson process can be constructed on the same probability space in such a way that the uniform distance between them is suitably controlled, and one can be use the former to approximate the latter trajectory by trajectory. Lemma \ref{lemma:MBePP} states this fact formally. It relies on the so called \emph{Hungarian construction} or KTM theorem of \cite{KomlosI}.
A second ingredient which is needed to prove Theorem \ref{diffapprox} is Lemma \ref{w-w} that we report below since it is needed for the proof of Theorem \ref{th:main} too.
Before stating this technical lemma we want to remark that the diffusion process \eqref{diff} has the same law of the solution of the following stochastic differential equation
\begin{equation}
G_\ast^{[n]}(t) = G_\ast^{[n]}(0)
+\sum_{l} \frac{l}{\sqrt{n}}\, \left[\sqrt{n}\int_{0}^{t} f_{l}(G_\ast^{[n]}(s))ds+ \int_{0}^{t} \sqrt{f_{l}(G_\ast^{[n]}(s))} dW_{l}(s)\right]\label{diff2}
\end{equation}
to which it is related to \eqref{diff} by a time-change (cf. Kurtz, Ch. 6, Sec. 5).

\begin{lemma}\label{lemma:MBePP} 
	Given a Brownian Motion $W_{t}$  we can construct a compensated Poisson process $\tilde N_{t}$ on the same probability space such that for any $\beta>0$ there exist positive constants $\lambda, \kappa$ and $c$ and a non-negative random variable $L_{n}$ such that
	\[\sup_{t\leq\beta n} \abs{\tilde N_{t}-W_{t}}\leq c \log n +L_{n}\]
	and
	\[
	\P [L_{n}>x]\leq k\, n^{-2} \text{e}^{-\lambda x}
	\]
\end{lemma}

\begin{lemma} \label{w-w}
	Let $\bar f_l$ be defined as in Theorem \ref{diffapprox}. We define
	\[\Lambda_{l,n}(\delta)=\sup_{\{u,h\geq 0 \; : \;u+h\leq n\bar f_{l}T, \;h\leq\delta \log n\}} \abs{W_{l}(u+h)-W_{l}(u)}\]
	
	and
	\[\bar \Lambda_{l,n}=\sup_{0\leq\delta\leq n \bar f_{l} T} (\delta+1)^{-\frac{1}{2}}\Lambda_{l,n}(\delta).\]
	
	There exist a non-negative random variable $L_{2,l,n}$ and positive constants $c_{2,l}$, $\kappa_{2,l}$, and $\lambda_{2,l}$ such that
	\[\bar \Lambda_{l,n}\leq c_{2,l} \log n +L_{2,l,n}\]
	and 
	\begin{align}
		\mathbb{P} \left( L_{2,l,n}  > x \right) \leq \kappa_{2,l}\, n^{-2} \exp \left( -\lambda_{2,l}\, x -\frac{x^2}{18 \log n} \right)
	\end{align}
\end{lemma}
	
The proofs of both theorems can be found in \cite{kurtz}[Ch. 11, Sec. 3], but be aware that the proof of Lemma \ref{w-w} is embedded in the proof of Theorem \ref{diffapprox}.

\section{Strong approximation in bounded domains} \label{robabella}
It is well documented by a large set of examples that if the state space is bounded and the boundaries are visited with non-negligible probability, the deterministic approximations of Section \ref{sec:background} give a largely inaccurate approximation of the original density dependent Markov chain even for moderately large values of the indexing parameter \cite{srivastava, ErdiLente, ATPN14, angius2015approximate}.
It is quite natural to think that the problem can be solved by adopting the sharper diffusion approximation of Theorem \ref{diffapprox}, but such an approximation would only hold up to the first hitting time of the boundary, then it becomes ill-posed.
In the following Subsection we are going to investigate in some more detail what exactly happens when the state space is bounded and in which sense the assumptions of Theorem \ref{diffapprox} prevent us to be satisfied with the diffusion approximation.

\subsection{On the inapplicability of the diffusion approximation}\label{inapplicability}
We believe it is worth to explain in some details what prevents the application of Theorem \ref{diffapprox} in practical cases from the mathematical point of view.
Let us focus on the one-dimensional case $d=1$ for simplicity, when the Markov chain $X^{[N]}$ is scalar. In almost all the examples of density dependent families \eqref{rates} the functions $f_l$ appearing in the rates \eqref{rates} are polynomials of degree 0, 1 or 2.

A first (minor) issue is that the Lipshitz conditions \eqref{Lip} fail at infinity if the polynomials have order higher then 1. This fact does not have relevant practical consequences since in Theorem \ref{diffapprox} the Lipschitz conditions \eqref{Lip} is only required to hold on an open interval $U\subset E^{[n]}$, and the upper extremal can be taken as large as we want so that even for moderate $n$, the probability that the process exit $U$ in a finite time $T$ is negligible.

 A second and much more important issue is that when the rate functions $f_l(x)$ are polynomials, often they are positive only for $x$ belong to a closed or half open interval.
 In this case the Markov chain $X^{[n]}$ is naturally confined in the domain where the rates are positive and the boundary of this domain can be visited with positive probability. The corresponding diffusion approximation, however can be more problematic: diffusion processes on closed intervals require a careful treatment of the behavior at the boundary. The next example, being very simple is a useful prototype of what can happen.
\begin{example}\label{esempioprimo}
Let $X^{[n]}(t)$ be a birth and death process with $X^{[n]}(0)=\frac{n}{2}$, linear birth rate $q^{[n]}_{k,k+1}=2 n\left(1-\frac{k}{n}\right)$ and linear death rate $q^{[n]}_{k,k-1}= k$. It is confined in $[0,n]$ and the diffusion approximation \eqref{diff2} of the corresponding density process $\frac{X^{[n]}}{n}(t)$ solves
\begin{equation}dG^{[n]}(t)= \left(2-3G^{[n]}(t) \right) dt+ \frac{1}{\sqrt{n}} \left[ \sqrt{2\left(1-G^{[n]}(t)\right)} dW_1(t) - \sqrt{G^{[n]}}dW_2(t)\right]\label{esempiofacile}\end{equation}
for some independent Brownian motions $W_1$ and $W_2$.
The following process, however
\begin{equation}d\tilde G^{[n]}(t)= \left(2-3\tilde G^{[n]}(t) \right) dt+ \sqrt{\frac{1}{n}\left(2-\tilde G^{[n]}(t)\right)} dW(t)\label{secver}\end{equation}
has the same generator, and hence, up to the first exit from the interval $(0,1)$, it also has the same trajectories as $G^{[n]}(t)$ (cf.\cite{allenEquivalent}).
Now this second process is naturally defined on $(-\infty,2)$ and, for finite $n$, it does not have any reason not to leave $(0,1)$, nor $[0,1]$. However right after a trajectory of \eqref{secver} exits $[0,1]$, the square roots in the coefficients of \eqref{esempiofacile} become negative and it cannot be a trajectory of \eqref{esempiofacile} any longer.
\end{example}

Of course, the theorem is not wrong and up to the first visit of the boundary the diffusion approximation holds. Moreover as $n\rightarrow\infty$ (but a sufficient $n$ might be really large) the noise, which scales with $\frac{1}{\sqrt{n}}$, becomes too weak to drive the process to the boundary in a finite time, and the diffusion approximation becomes feasible. In many realistic settings however, boundaries are often visited with non-negligible probability also in finite time horizons and in this case the diffusion approximation is of no help.

The interest in Example \ref{esempioprimo} comes from the fact that it is so simple that can be tackled on a purely analytical ground.
Indeed, much more complex and interesting examples can be exhibited (and actually have been exhibited in the references cited above), and examples where the problem is more dramatically relevant (see e.g. Section \ref{nasty}) are known. However, in most cases only simulations are feasible and it becomes less obvious that a theoretical problem is really there that it is not just due some numerical artifact.

\subsection{The new jump-diffusion approximation}

In his famous paper \cite{feller1d}, Feller introduced the so called elementary return process. It is a Markov process on a closed interval which behaves as a diffusion up to the first time it hits one of the endpoints. At the boundary it stays trapped for an exponential time and then it jumps back to the interior of the interval with a given jump distribution. Back in the interior, it starts diffusing anew with the same law.
This kind of process provides us with the right tool to develop a strong approximation for  one-dimensional density dependent families of Markov chains on a closed interval $[l,u]$. Indeed we can approximate the density process of the original Markov chain with its diffusion approximation as long as the latter remains in the interior. When the diffusion process reaches one of the boundaries we mimic the behavior that the discrete density process would have had there: we make the process wait an exponential time at $l$ (or $u$) with the same parameter the density process would have, and then a jump to the point $l+\frac{1}{n}$ (or to $u-\frac{1}{n}$).

In this Section we aim at discussing how accurate such an approximation is and at showing that a suitably defined jump-diffusion process that generalizes this behavior to multidimensional bounded state spaces can be defined. We prove a strong approximation theorem for density dependent families of Markov chains defined on bounded (multidimensional) state spaces. The approximation rate we achieve is the same as those achieved by diffusion approximation in Theorem \ref{diffapprox}, but it still holds when the boundaries are reached.
To avoid an unnecessarily heavy notation, we prove the approximation under the simplified hypothesis that the Lipschitz conditions \eqref{Lip2} hold in the whole state space $E^{[n]}$. As already mentioned at the beginning of Section \ref{inapplicability} such an assumption is often not satisfied (e.g. when the rate functions are polynomials of degree 2, cf. Section \ref{nasty}). However in this case a slightly modified version of the theorem still holds (cf. Remark \ref{relax}) and the approximation can be safely used for every finite time horizon.

\begin{theorem}\label{th:main}
	Let $X^{[n]}(t)$ be as in \eqref{mc} and such that 
	\begin{equation}\bar f_l=\sup_{x \in E} f_l(x)<\infty\text{ and }\bar f_l=0\label{boundedf}\end{equation}
	except for finitely many $l$. Suppose $M>0$ satisfies both the two equations below for any $x,y\in E$
	\begin{equation}\begin{aligned}
	\abs{f_l(x)-f_l(y)}&\leq M\abs{x-y}\\
	\abs{F(x)-F(y)}&\leq M\abs{x-y}.
	\end{aligned}\label{Lip2}\end{equation}
	
Define a jump diffusion process
\begin{align}\label{eq:main:Z}
Z^{[n]}(t) =& X^{[n]}(0)+ \sum_{l}l \left\{\int_{0}^{t} \left[1- \mathrm{ab} \left(l,Z^{[n]}(t)\right)\right] f_{l}\left(Z^{[n]}(s)\right)ds\right.\\
&+\frac{1}{n}\, W_{l}\left(n \int_{0}^{t} \left[1-\mathrm{ab} \left(l,Z^{[n]}(t)\right)\right] f_{l}\left(Z^{[n]}(s)\right)ds\right) +\\
&+\left.\frac{1}{n}  N_{l}\left(n \int_{0}^{t}  \mathrm{ab} \left(l,Z^{[n]}(t)\right)\, f_{l}\left(Z^{[n]}(s)\right)ds\right)\right\},
\end{align}
	where the $W_l(t)$ are independent standard Wiener processes and the function $\mathrm{ab}(l,x)$ is defined as follows
	\[\mathrm{ab}(l,x)=\mathbb{1}\big\{ \exists i : x_{i} \in \partial E^{[n]}_{i} \text{\rm and } (x_{i}+n^{-1}l_{i}) \notin \partial E_{i} \big\}.\]
	The function $\text{ab}(l,x)$ takes value $1$ or $0$ depending if the increment $l$ moves at least a component away from the boundary (ab) or not.

	Fix  $T>0$. Given a weak solution $\left(\{N_l(t)\}_l,\{W_l(t)\}_l,Z^{[n]}(t)\right)$ of equation \eqref{eq:main:Z}  we can construct on the same probability space a  stochastic process $\hat X^{[n]}(t)$ such that
	\begin{enumerate}
		\item $\hat X^{[n]}$ has the same law of $X^{[n]}$ (cf. equation \eqref{mc})
		\item  for $n\geq 2$  there exist a positive constant $c_{T}$ and a random variable $L_{n,T}= o_{P}(1)$ such that 
 \begin{align}\label{eq:main:bound}
 \frac{n}{\log n}\sup_{t\leq T}\, \abs{\hat X^{[n]}(t)-  Z^{[n]}(t) }\leq c_{T}+L_{n,T}
 \end{align}
and in particular there exist positive constants $\lambda_{T}$ and $k_{T}$ such that
\[\P \left(L_{n,T}>x \right)\leq k_{T} n^{-2} \exp\left(- \lambda_{T}x^{\frac{1}{2}}-\frac{\lambda_{T}x}{\log n}\right)\]
\end{enumerate}
\end{theorem}

\begin{proof}
Let us define	$\hat X^{[n]}$ so as to solve
\[\begin{aligned}\label{eq:main:hatX}
\hat X^{[n]}(t) =& X^{[n]}(0)+\sum_l \frac{l}{n}\,\left\{  N^W_{l}\left(n \int_{0}^{t} \left[1-\text{ab}\left(l,Z^{[n]}(t)\right)\right] f_{l}\left(\hat X^{[n]}(s)\right)ds\right) +\right.\\
&+\left.  N_{l}\left(n \int_{0}^{t}  \text{ab}\left(l,Z^{[n]}(t)\right)\, f_{l}\left(\hat X^{[n]}(s)\right)ds\right)\right\}=\\
=& X^{[n]}(0)
+\sum_{l} l \left\{\int_{0}^{t} \left[1-\text{ab} \left(l,Z^{[n]}(s)\right)\right]  f_{l}\left(\hat X^{[n]}(s)\right)\, ds+\right.\\
&\left.+\frac{1}{n}\, \tilde N^W_{l}\left(n \int_{0}^{t} \left[1-\text{ab}\left(l,Z^{[n]}(t)\right)\right] f_{l}\left(\hat X^{[n]}(s)\right)ds\right) +\right.\\
&+\left.\frac{1}{n}  N_{l}\left(n \int_{0}^{t}  \text{ab}\left(l,Z^{[n]}(t)\right)\, f_{l}\left(\hat X^{[n]}(s)\right)ds\right)\right\},
\end{aligned}\]
where each of the $\tilde N^W_{l}$ is the compensated Poisson process constructed from the Wiener process $W_l$ by Lemma \ref{lemma:MBePP}.
Let us notice that the Poisson processes $N_{l}$ are not compensated since it will be easier to handle them in this form.

The process $\hat X^{[n]}$ has the same law as $X^{[n]}$.	Let us define $\Delta_{n}(t)=\abs{\hat X^{[n]}(t)-Z^{[n]}(t) }$ and $\bar\Delta_{n,T}=\sup_{t\leq T}\abs{\hat X^{[n]}(t)-Z^{[n]}(t) }$. By \eqref{eq:main:hatX} and \eqref{eq:main:Z} we have

\begin{align}
\Delta_{n}(t)&=\left| \sum_{l}\frac{l}{n}\left\{\left[\tilde N^W_{l}\left(n \int_{0}^{t} \left[1-\text{ab}\left(l,Z^{[n]}(s)\right)\right]  f_{l}\left(\hat X^{[n]}(s)\right)\,ds\right)-\right.\right.\right.\\
&\phantom{=} \left.- \,W_{l}\left(n \int_{0}^{t} \left[1-\text{ab}\left(l,Z^{[n]}(s)\right)\right]  f_{l}\left(Z^{[n]}(s)\right)\,ds\right)\right] + \notag \\
&\phantom{=}+ \left[N_{l}\left(n \int_{0}^{t} \text{ab}\left(l,Z^{[n]}(s)\right)\, f_{l}\left(\hat X^{[n]}(s)\right)\,ds\right)-\right. \\
&\phantom{=} \left. -\, N_{l}\left(n \int_{0}^{t}\text{ab}\left(l,Z^{[n]}(s)\right)\,  f_{l}\left(Z^{[n]}(s)\right)\,ds\right) \right]+ \notag \\
&\phantom{=} + \left. \left. n\int_{0}^{t} \left[1-\text{ab}\left(l,Z^{[n]}(s)\right)\right] \left[ f_l \left(\hat X^{[n]}(s)\right)- f_l \left(Z^{[n]}(s)\right) \right] \, ds \right\} \right| \leq \notag \\
\end{align}
and hence
\begin{align}
\Delta_{n}(t)\leq&C_n(t)+\sum_{l}\frac{\abs{l}}{n}\left[A_{l,n}(t) +B_{l,n}(t)+ D_{l,n}(t)\label{DeltaLeq}\right] 
\end{align}
where
\[\begin{aligned}{}&A_{l,n}(t)=&&\left| \tilde N_{l}^W\left(n \int_{0}^{t} \left[1-\text{ab}\left(l,Z^{[n]}(s)\right)\right] f_{l}\left( \hat X^{[n]}(s)\right)\,ds \right) -\right. \\
&&&\left.\qquad-W_{l}\left(n \int_{0}^{t}  \left[1-\text{ab}\left(l,Z^{[n]}(s)\right)\right] f_{l}\left(\hat X^{[n]}(s)\right)ds\right)\right|\\
&B_{l,n}(t)=&&\left| W_{l}\left(n \int_{0}^{t} \left[1-\text{ab}\left(l,Z^{[n]}(s)\right)\right]  f_{l}\left(\hat X^{[n]}(s)\right)\,ds\right) -\right.\\
&&&\left.\qquad-W_{l}\left(n \int_{0}^{t} \left[1-\text{ab}\left(l,Z^{[n]}(s)\right)\right]  f_{l}\left(Z^{[n]}(s)\right)ds\right)\right| \\
&C_{n}(t)=&& \int_{0}^{t} \abs{F(\hat X^{[n]}(s))-F(Z^{[n]}(s))}\, ds \\
&D_{l,n}(t)=&&\left| N_{l}\left(n \int_{0}^{t} \text{ab}\left(l,Z^{[n]}(t)\right) f_{l}\left(\hat X^{[n]}(s)\right)\,ds\right)-\right.\\
&&&\left. \qquad -N_{l}\left(n \int_{0}^{t} \text{ab}\left(l,Z^{[n]}(t)\right)f_{l}\left(Z^{[n]}(s)\right)ds\right)\right| \end{aligned}\]

Let us discuss how to estimate each term in the previous equations. We have
\[A_{l,n}(t)\leq \sup_{0\leq t\leq n\bar f T}\abs{\tilde N_{l}(t)-W_{l}(t)}\]
and therefore by Lemma \ref{lemma:MBePP}, there exist positive constants $c_{1,l}$, $\kappa_{1,l}$ and $\lambda_{1,l}$ and a non negative random variable $L_{1,l,n}$ such that
\[\begin{aligned}A_{l,n}(t)\leq c_{1,l} \log n+L_{1,l,n},
\end{aligned}\]
and

\begin{equation}P(L_{1,l,n}>x)\leq \kappa_{1,l}\, n^{-2}\text{e}^{-\lambda_{1,l}\,  x}\label{primaqua}\end{equation}

Let us focus on the summand $B_{l,n}(t)$. Let us define\footnote{To get an intuitive idea of why in the following definition the range of the variables over which the $\sup$ is taken is the one we propose, the reader can either go through the whole proof and accept ex-post that they lead to the right result, or appreciate the following heuristics. Thanks to the hypothesis \eqref{boundedf} the two time instants at which Brownian motion is evaluated are both bounded between 0 and $n\bar f_{l}T$. By the Lipschitz property \eqref{Lip2} their difference is bounded by $M T \cdot n\bar \Delta_{n,T}$. If the the theorem were true, we would have $\bar\Delta_{n,T}=O_P\left( \frac{\log n}{n}\right)$ so that for large $n$ such a difference, that denote by $h$ would be of order $\log n$.}
\[\Lambda_{l,n}(\delta)=\sup_{\{u,h\geq 0 \text{ s.t. } u+h\leq n\bar f_{l}T\; \& \:h\leq\delta \log n\}} \abs{W_{l}(u+h)-W_{l}(u)}\]

and
\[\bar \Lambda_{l,n}=\sup_{0\leq\delta\leq n \bar f_{l} T} (\delta+1)^{-\frac{1}{2}}\Lambda_{l,n}(\delta).\]

We have
\begin{align*}
{B_{l,n}(t)}&\leq \Lambda_{l,n}\left(\frac{n}{\log n}\left|\int_{0}^{t} \left[1-\text{ab}\left(l,Z^{[n]}(s)\right)\right] f_{l}\left(\hat X^{[n]}(s)\right)\,ds-\right.\right.\\
&\qquad \qquad\left.\left.-\int_{0}^{t} \left[1-\text{ab}\left(l,Z^{[n]}(s)\right)\right]f_{l}\left(Z^{[n]}(s)\right)ds\right|\right)\leq\\
&\leq \Lambda_{l,n}\left(\frac{n}{\log n}\int_{0}^{t} \left| f_{l}\left(\hat X^{[n]}(s)\right)- f_{l}\left(Z^{[n]}(s)\right)\right| \,ds\right)\leq\\
&\leq \Lambda_{l,n}\left(\frac{nM}{\log n} \int_{0}^{t} \left|\hat X^{[n]}(s)- Z^{[n]}(s)\right|\,ds\right)\leq&\\
&\leq \Lambda_{l,n}\left(\frac{nM}{\log n} \int_{0}^{t} \Delta_{n}(s)\,ds\right)\leq \left(1+ \frac{nM}{\log n}\int_{0}^{t}\Delta_{n}(s)\,ds\right)^{\frac{1}{2}}\bar\Lambda_{l,n}
\end{align*}

By Lemma \ref{w-w}, there exist positive constants $c_{2,l}$, $\kappa_{2,l}$ and $\lambda_{2,l}$ and a non negative random variable $L_{2,l,n}$ such that

\[\begin{aligned}B_{l,n}(t)\leq \left(1+ \frac{nM}{\log n}\int_{0}^{t}\Delta_{n}(s)\,ds\right)^{\frac{1}{2}}\left( c_{2,l} \log n+L_{2,l,n} \right)
\end{aligned}\]

and

\begin{align}
\mathbb{P} \left( L_{2,l,n}  > x \right) \leq \kappa_{2,l}\, n^{-2} \exp \left( -\lambda_{2,l}\, x -\frac{x^2}{18 \log n} \right)\label{secondaqua}
\end{align}

To estimate ${C_{n}(t)}$, it is sufficient to apply the Lipschitz property \eqref{Lip2} and we get
\[{C_{n}(t)}\leq M\int_{0}^{t}\Delta_{n}(s)\,ds.\]

Our next step is to deal with ${D_{l,n}(t)}$. Given the random variables
\[\Pi_{l,n}(\delta)=\sup_{\{u,h\geq 0 \text{ s.t. } u+h\leq n\bar f_{l}T\; \& \:h\leq\delta \log n\}} \abs{N_{l}(u+h)-N_{l}(u)}\]
and
\[\bar \Pi_{l,n}=\sup_{0\leq\delta\leq n \bar f_{l} T} \Pi_{l,n}(\delta),\]We have
\[ \begin{aligned}
{D_{l,n}(t)}&\leq \Pi_{l,n}\left(\frac{n}{\log n}\left|\int_{0}^{t} \text{ab}\left(l,Z^{[n]}(t)\right) f_{l}\left(\hat X^{[n]}(s)\right)\,ds-\right.\right.\\
&\qquad\qquad\qquad\left.\left. \int_{0}^{t} \text{ab}\left(l,Z^{[n]}(t)\right) f_{l}\left(Z^{[n]}(s)\right)ds\right|\right)\leq\\
&\leq \Pi_{l,n}\left(\frac{n}{\log n}\int_{0}^{t} \left| f_{l}\left(\hat X^{[n]}(s)\right)- f_{l}\left(Z^{[n]}(s)\right)\right| \,ds\right)\leq\\
&\leq \Pi_{l,n}\left(\frac{nM}{\log n} \int_{0}^{t} \left|\hat X^{[n]}(s)- Z^{[n]}(s)\right|\,ds\right)\leq&\\
&\leq \Pi_{l,n}\left(\frac{nM}{\log n} \int_{0}^{t} \Delta_{n}(s)\,ds\right)\leq \bar\Pi_{l,n}
\end{aligned}
\]

We claim that (a proof will be given separately in Lemma \ref{NmenoN}) there exist two positive constants $c_{3,l}$, $\kappa_{3,l}$ and a non-negative random variable $L_{3,l,n}$ such that
\[D_{l,n}(t)\leq \bar \Pi_{l,n}\leq c_{3,l} \log n + L_{3,l,n}
\]

and 

\begin{equation}P (L_{3,l,n}>  x)\leq \kappa_{3,l}\frac{1}{n^{2}}\text{e}^{-x/2}.\label{terzaqua}\end{equation}


and from equation \eqref{DeltaLeq} we get
\[\begin{aligned}{}&\Delta_{n}(t)\leq M\int_{0}^{t}\Delta_{n}(s)\,ds\\
&+\sum_{l} \frac{\abs{l}}{n}\left[ (c_{1,l}+c_{3,l}) \log n+L_{1,l,n}+ L_{3,l,n} +\left(1+ \frac{nM}{\log n}\int_{0}^{t}\Delta_{n}(s)\,ds\right)^{\frac{1}{2}}\left( c_{2,l} \log n+L_{2,l,n} \right) \right] 
\end{aligned}\]
and, defining $c_{4,l}=c_{1,l}+c_{3,l}$ and $L_{4,l,n}=L_{1,l,n}+ L_{3,l,n}$, by Gronwall's  inequality  and by taking a supremum over $0\leq t \leq T$ and multiplying by $\frac{n}{\log n}$, we have
\[\begin{aligned}\frac{n}{\log n}\bar\Delta_{n}\leq &\text{e}^{MT}\sum_{l}\abs{l}\left[ c_{4,l}+\frac{L_{4,l,n}}{\log n}+ \left(\frac{1}{MT}+\frac{n}{\log n}\bar\Delta_{n}\right)^{\frac{1}{2}}\sqrt{MT} \left( c_{2,l}+\frac{L_{2,l,n}}{\log n} \right) \right].
\end{aligned}\]
As $y\leq a +b \sqrt{y}$ implies $y\leq 2a+b^{2}$, by taking $y=\frac{1}{MT}+\frac{n}{\log n}\bar\Delta_{n}$  we can show that 
\begin{align} \frac{n}{\log n}\bar\Delta_{n}\leq &\frac{1}{MT}+2\text{e}^{MT}\sum_{l}\abs{l}\left[ c_{4,l}+\frac{L_{4,l,n}}{\log n}+ MT \left( c_{2,l}+\frac{L_{2,l,n}}{\log n} \right)^{2} \right].\label{uu}
\end{align}
We  define $c_T$ as the sum of all constants in the right hand side of the inequality \eqref{uu} (which are all positive), and $L_{n,T}$ as the sum of all (proper) random variables (which are all non-negative) and rewrite equation \eqref{uu} as
\[ \frac{n}{\log n}\bar\Delta_{n}\leq  c_{T}+L_{n,T}.\]
Since $L_{n,T}$ is a finite sum, taking into account equations \eqref{primaqua}, \eqref{secondaqua} and \eqref{terzaqua}, and the inequality $P(X+Y\geq x)\leq P(X\geq x/2)+P(Y \geq x/2)$ that holds for any couple of random variables $X$ and $Y$, we have that there exist constants $\kappa_T$ and $\lambda_T$ such that
\[P(L_{n,T}<x)\leq \kappa_{T}\,  n^{-2} \text{exp}\left( -\lambda_T \sqrt{(x)} - \frac{\lambda_T\, x}{\log n}\right).\]

\end{proof}

\begin{remark}\label{relax}
The hypothesis \eqref{boundedf} and \eqref{Lip2} are given for $x,y \in E^{[n]}$ but they are too strong to include even very simple examples. Actually the result can be reformulated under the following more general setting. For fixed $K$ let $N_K =\{y\in E:\inf_{t\leq T}\abs{x(t)-y}\leq K\}$ and $\tau_n=\inf\{t:\hat X^{[n]}(t)\notin N_K \textup{ or }Z^{[n]}(t)\notin N_K\}$, hypothesis \eqref{boundedf} and \eqref{Lip2} can be reformulated as
	\begin{align*}
	&\bar f_l=\sup_{x \in N_\varepsilon} f_l(x)<\infty\\
	&\abs{f_l(x)-f_l(y)}\leq M\abs{x-y}, \quad x,y \in N_\varepsilon\\
	&\abs{F(x)-F(y)}\leq M\abs{x-y}, \quad x,y \in N_\varepsilon.
	\end{align*}
and equation \eqref{eq:main:bound} in the thesis becomes
 \begin{align*}
 \frac{n}{\log n}\sup_{t\leq T \wedge \tau_n}\, \abs{\hat X^{[n]}(t)-Z^{[n]}(t) }\leq c_{T}+L_{n,T}.
 \end{align*}
  For $n \rightarrow \infty$ we have $\mathbb{P}(\tau_n>T)\rightarrow1$.
  Moreover in most practical cases $K$ can be arbitrarily large (the Lipschitz condition is violated only at infinity): in this case we can choose it so large to guarantee that $\mathbb{P}(\tau_n<T)$ is negligible for any finite $n$ and for any finite time horizon $T$. 
\end{remark}

\begin{lemma}
	Let $X$ be Poisson distributed with intensity $\lambda$ and let $x \geq 1$. The following bound for the right tail probability holds true
	\[P(X>x)< \left(1-\text{e}^{-\lambda}\right)^{x}.\]
\end{lemma}
\begin{proof}
	Recalling that a Poisson$(\lambda)$ r.~v. is identically distributed as $N_1$, where $(N_t)_{t \geq 0}$ is a Poisson process with intensity $\lambda$, the following inequality can be easily derived
	\[ P(X>x) \leq P(Y\leq \lambda),\]
	where $Y$ is Gamma distributed with shape parameter $x$ and unitary scale. The inequality reduces to an equality for integer $x$. From \cite{alzer} eq. (2.6), with a simple change of variable\footnote{using the parameter $p$ and the variable $t$ with the same notation as in \cite{alzer} you should take $x=t^p$, while $\frac{1}{p}$ becomes the shape parameter.}, we have
	\[P(Y\leq \lambda)<\left(1-\text{e}^{-\lambda}\right)^{x},\]
	as far as $x\geq1$.
\end{proof}

\begin{lemma}\label{NmenoN}
	Given the random variables
	\[\Pi_{l,n}(\delta)=\sup_{\{u,h\geq 0 \text{s.t. } u+h\leq n\bar f_{l}T\; \& \:h\leq\delta \log n\}} \abs{N_{l}(u+h)-N_{l}(u)}\]
	and
	\[\bar \Pi_{l,n}=\sup_{0\leq\delta\leq n \bar f_{l} T} \Pi_{l,n}(\delta),\]
	there exist positive constants $c_{3,l}$ and $\kappa_{3,l}$ such that
	\[ P (\bar \Pi_{l,n}> c_{3,l} \log n + x)\leq \kappa_{3,l}\frac{1}{n^{2}}\text{e}^{-x/2}.\]
	Defining the non negative r.v.
	\[L^{3}_{l,n}=(\bar \Pi_{l,n}- c_{3,l}\log n)_{+} \]
	we have
	\begin{equation} P (L^{3}_{l,n}>  x)\leq \kappa_{3,l}\frac{1}{n^{2}}\text{e}^{-x/2}.\label{L3}\end{equation}
\end{lemma}
\begin{proof}
	Suitably subdividing the interval we can force the $\sup$ over the variable $u$ to run only over integers. We get for $k_{n}=\left\lfloor\frac{n\bar f T}{\delta \log n}\right\rfloor+1$
	\[\Pi_{l,n}(\delta) \leq 2 \sup_{\{k\leq k_{n} \& \:h\leq\delta \log n\}} \abs{N_{l}(k\delta\log n+h)-N_{l}(k\delta\log n)}.\]
	Now, the probability that one of the arguments of the sup is larger than a given value is lesser or equal to the probability that at least one the arguments is larger then the same value, by  sub-additivity. Moreover, the arguments of the sup are stationary increments. Thus all such probabilities are the same, then
	
	\[\begin{aligned}P \left(\Pi_{l,n}(\delta)> c \log n + x\right) &\leq k_{n} \; P\left[ \sup_{h\leq \delta \log n} \abs{N_{l}(h)} > \frac{1}{2}( c \log n + x)\right] \\
	&=k_{n}\; P \left[  N_{l}(\delta \log n) > \frac{1}{2}( c \log n + x )\right] \\
	&\leq k_{n} (1-\text{e}^{-\delta \log n})^{(c \log n + x)/2}\\
	&= k_{n} \text{e}^{\log(1-n^{-\delta}) (c\log n +x)/2}\\
	&\leq  k_{n}\text{e}^{-\frac{c\log n +x}{2 n^{\delta}}}\\
	&\leq k_{n}\text{e}^{-(c\log n +x)/2}
	\end{aligned}\]
	and since $c$ is arbitrary and $k_{n}=O(n)$, there exist two positive constants $c_{3,l}$ and $\kappa_{3,l}$ such that for every positive $\delta$
	\[P \left(\Pi_{l,n}(\delta)> c_{3,l} \log n + x\right) \leq \kappa_{3,l} \frac{1}{n^{2}}\text{e}^{-x/2}\]
	and since the bound is uniform in $\delta$ we also have
	\[P \left(\bar \Pi_{l,n}> c_{3,l} \log n + x\right) \leq \kappa_{3,l} \frac{1}{n^{2}}\text{e}^{-x/2}\]
	and \eqref{L3} immediately follows.
	%

\end{proof}

\section{Reaction networks and examples}\label{nasty}

A prominent application of density dependent families of Markov chains is to model reaction networks (cf. \cite{andersonKurtzBook,andersonKurtz, kurtzbook}).
A reaction network is a triple $\{\mathcal{S},\mathcal{C}, \mathcal{R}\}$ where
\begin{enumerate}
	\item $\mathcal{S}=\{S_1,\cdots, S_d\}$ is the set of species of cardinality $d$.
	\item $\mathcal{C}$ is the set of complexes, consisting of some nonnegative linear combination of the species
	\item $\mathcal{R}$ that is a finite set of ordered couples of complexes that are often rendered graphically by  arrows.
	\end{enumerate}
Reaction networks are usually specified by writing the corresponding stoichiometric equations.
For each reaction we write
\begin{equation}\sum c_{i} S_i \rightarrow \sum {c'}_{i} S_i\label{reactionk}\end{equation}
if such reaction consumes the complex $\sum c_{i} S_i$ to produce the complex $\sum {c'}_{i} S_i$. The current state of the network is encoded into the vector $s=(s_1,\cdots,s_d)$ which counts how many molecules of each of the species are available in the system and the  reaction \eqref{reactionk} has the effect of updating the state of the network to $s+l=(s_1+l_{1},\cdots,s_n+l_{n})$ where each state-increment vector $l$ can be calculated as $l={c'}-c$.

The reactions are assumed to be events of a Markov chain. In the case when each reaction causes a different state-increment, we identify the reactions with their state increment and we say that the network follows mass-action kinetics if the rate of each reaction is in the form
\begin{equation}q_{s,s+l}=\frac{\lambda_l}{V^{\langle c \rangle-1}}\prod_{i=1}^d \binom{s_i}{c_{i}}= V\left[\frac{\lambda_l}{\prod_{i=1}^d c_i !}\prod_{i=1}^d \left(\frac{s_i}{V}\right)^{c_i} +O\left(\frac{1}{V}\right)\right] \label{massaction}\end{equation}
where, for vectors $v\in \mathbb{R}^d$ we denoted by $\langle v \rangle$ the scalar $\langle v \rangle=\sum_i v_i$, and where $\lambda_l$ are arbitrary constants related to the so called propensity of each reaction (the higher the constant the more frequent is the reaction).

In some case there might be different reactions (say 2), both with mass-action kinetics, that cause the same state increment. In such a case there is no need to distinguish which one of the two reaction is really occurring, and the rate related to the state increment is the sum $q_{s,s+l_k}=q^1_{s,s+l_k}+q^2_{s,s+l_k}$ of the individual rates $q^1_{s,s+l_k}$ and $q^2_{s,s+l_k}$ of the two reactions, both sharing mass-action form \eqref{massaction}.

As apparent from condition \eqref{massaction}, mass-action  rates are, at least approximatively, in the density dependent form \eqref{rates}, where the the indexing parameter denoted by $[n]$  in equation\eqref{rates}  is played by the volume of reaction $V$. The density process \eqref{density} is then a model for concentrations and can be approximated in the different ways we have illustated.

In the diffusion approximation \eqref{diff} of Chapter \ref{sec:background} the different reactions sharing the same  state increment $l$ are still notationally identified by $l$. The effect of each reaction (or reactions group), however, is not anymore encoded in a sudden state change (jump), but rather affects the state of the process continuously in time by an infinitesimal change proportional to $l$. In particular, each summand in equation \eqref{diff} accounts for the specific effect of the reaction (or reactions group) $l$.

In modeling reaction networks, concentrations are  not allowed to become negative and moreover the conservation of the total mass in most cases imposes an upper bound to the concentrations which is determined by the initial condition. It has been already noticed that the diffusion approximation may fail to fulfill such constraints, while our new jump-diffusion process \eqref{eq:main:Z} can be safely adopted.

The summands in equation \eqref{eq:main:Z} still model the effect of each single reaction (of group of reactions with same effect). If no chemical species has null or maximal concentration the effect of the reactions are continuously compounded. However as soon as the concentration of a species becomes null (or maximal), only those reactions whose effect is to increase (decrease) the concentration of such species (so that it can leave the boundary) start to act discretely again (by jumps). Meanwhile, all the other reactions that do not affect the concentration of the specific chemical having vanishing (or maximal) concentration are still continuously compounded. Once the boundary is left by a jump, all the related reactions return to have a continuously compounded effect.

\subsection{Simulations}

 A simulation algorithm for the approximating jump-diffusion process is exposed in \cite{angius2015approximate, ATPN14}). A slightly naive version of it is
 \begin{itemize}
 	\item fix a provisional discretization step $\delta$
 	\item check what components of the process are at the boundary (if any) and consequently decide which are the reactions to by simulated discretely
 	\item simulate the time $\tau$ at which the first reaction occur
 	\item if $\tau<\delta$ then update the state accounting for the effect of the reaction 
 	\item account for the effect of the continuously approximated reactions with an Euler approximation with step $\min(\tau,\delta)$.
 	\item iterate.
 	\end{itemize}
 
The case studies addressed in \cite{angius2015approximate, ATPN14} give a large body of evidence that simulating the approximating process $Z^{[V]}(t)$ with this naive algorithm or with some obvious step-adaptation is in most cases much faster then simulating the process $Y^{[V]}(t)$ directly. 
We do not want to add much here, but we warn the reader that in some nasty example (as the two introduced in the next subsections) an Euler method with a fixed step-size may not be suitable unless more sophisticated considerations are taken into account. Indeed:
\begin{itemize} 
	\item when the process is close to the boundary, a very small discretization step $\delta$ may be needed in order to avoid missing some crossings of the boundaries that could cause a large loss of precision
	\item if the process spends almost all of the time at the boundary, the real step at which our algorithm  proceeds can in some case be as short as that of the original chain with an overload due to the decision of which of the reactions act discretely and to the  simulation of the Brownian motions for the continuous ones.
\end{itemize}
While the first issue may be addressed by some better simulation algorithm (cf. \cite{baldi2014}), the second is structural, but it is important only if the reactions which may lead the process out of the boundary are among the fastest in the network.
We remark that this situations are not so common and that in most examples (cf. \cite{angius2015approximate, ATPN14,new}) the computational cost of our algorithm is much advantageous with respect to the standard stochastic simulation algorithms for Markov chains. Moreover such an algorithm is still subject to many possible improvements, such as simulating an equivalent diffusions with the minimal number of Brownian motion (cf.\cite{allenEquivalent}), and applyig the results of \cite{baldi2014} and \cite{gobet} which could allow us to use  a larger Euler step even close to the boundaries without loosing accuracy. We also remark that speeding up simulations is not the only goal of an approximation method, e.g. in \cite{angius2015approximate} we also used numerical solution of the Fokker-Plank equation of the approximating process in a couple of one-dimensional examples.
In the example that follows we do not take care at all of the computational cost, but simply evaluate the precision of the approximation in some complicated example where the process stays in the vicinity of the boundaries for most of the time. We choose very small simulation steps, regardless to the fact that in those examples the time required for the simulations can be even longer than that for the Markov chain with the aim of demonstrating that even is such.

\subsection{A nasty example}

The aim of this Section is to test the validity of the approximation in a particularly hard example that combines non-linearity and stiffness and for which the very frequent visits to the boundary of the state space,  have a dramatic impact on the overall dynamics of the system.
The reaction network was first described in \cite{TogashiPRL} and the authors explicitly point out that not only the deterministic continuous approximation fail to catch the behavior of the system but that it cannot be described by diffusion processes as well due to the effect of the so called Discretely Induced Transitions that we are going to describe below). The network is composed by four chemical species $S_i$ with $Y_i$ molecules each, $i\in \{1,\ldots,4\}$ that undergo a loop of autocatalytic reactions
\begin{equation} S_i+S_{(i+1) \text{ mod } 4}\stackrel{1}{\longrightarrow} 2 S_{(i+1) \text{ mod } 4} \label{ciclo} \end{equation}

within a container which is kept in contact with a reservoir of each of the chemical species at constant concentration with molecules flowing in and out according to
\begin{equation} S_i\stackrel{D}{\longrightarrow} \emptyset\label{out}\end{equation}
and
\begin{equation} \emptyset \stackrel{D}{\longrightarrow} S_i.\label{in}\end{equation}

The volume of the container is denoted by $V$ and the process $X^{[V]}(t)=\frac{1}{V}\Big(Y_1(t),Y_2(t),Y_3(t),Y_4(t)\Big)$ of the concentrations at time $t$ belongs to a density dependent family of MCs indexed by $V$. The state space is a discrete lattice, subset of the positive orthant $\mathbb{R}_+^4$ (including the hyperplanes when one or more coordinates are vanishing), that becomes thicker when $V\rightarrow\infty$. In Table \ref{tabella} for each of the 12 reactions we display the increments $l$, the reaction rates $q_{y,y+l}$ and the functions $f_l(x)$ of equation \eqref{rates}. 

\begin{table}[h]

\setlength{\tabcolsep}{0pt}
\renewcommand{\arraystretch}{1.2}

\begin{tabular}{c|ccccccccccccc}\hline
	\text{reac.}&\phantom{$I$}&1&2&3&4&5&6&7&8&9&10&11&12\\ \hline\hline
	$l$&\phantom{$\Big|^\Big|$}\hspace{-2mm}&$\left(\hspace{-.7mm}\begin{smallmatrix}-1\\1\\0\\0\end{smallmatrix}\right)$
	& $\left(\hspace{-.7mm}\begin{smallmatrix}0\\-1\\1\\0\end{smallmatrix}\right)$
	& $\left(\hspace{-.7mm}\begin{smallmatrix}0\\0\\-1\\1\end{smallmatrix}\right)$
	& $\left(\hspace{-.7mm}\begin{smallmatrix}1\\0\\0\\-1\end{smallmatrix}\right)$
	& $\left(\begin{smallmatrix}1\\0\\0\\0\end{smallmatrix}\right)$
	& $\left(\begin{smallmatrix}0\\1\\0\\0\end{smallmatrix}\right)$
	& $\left(\begin{smallmatrix}0\\0\\1\\0\end{smallmatrix}\right)$
	& $\left(\begin{smallmatrix}0\\0\\0\\1\end{smallmatrix}\right)$
	& $\left(\hspace{-.7mm}\begin{smallmatrix}-1\\0\\0\\0\end{smallmatrix}\right)$
	& $\left(\hspace{-.7mm}\begin{smallmatrix}0\\-1\\0\\0\end{smallmatrix}\right)$
	& $\left(\hspace{-.7mm}\begin{smallmatrix}0\\0\\-1\\0\end{smallmatrix}\right)$
	& $\left(\hspace{-.7mm}\begin{smallmatrix}0\\0\\0\\-1\end{smallmatrix}\right)$\\
		$q_{y,y+l\phantom{u}}$&$\phantom{\Bigg|^\big|}$&$\displaystyle{\frac{y_1 y_2}{V}}$ & $\displaystyle{\frac{y_2 y_3}{V}}$ & $\displaystyle{\frac{y_3 y_4}{V}}$ & $\displaystyle{\frac{y_4 y_1}{V}}$ & $V\hspace{-.7mm}D$ & $V\hspace{-.7mm}D$ & $V\hspace{-.7mm}D$ & $V\hspace{-.7mm}D$ & $ D y_1$&$ D y_2$&$ D y_3$&$ D y_4$\\
	$f_l(x)$& & $x_1 x_2$ & $x_2 x_3$ & $x_3 x_4$ & $x_4 x_1$ & $D$ & $D$ & $D$ & $D$ & $x_1$&$x_2$&$x_3$&$ x_4$ \\ \hline
\end{tabular}
\caption{The network of reactions \eqref{ciclo}, \eqref{in} and \eqref{out} as a density dependent family of MCs\label{tabella}}
\end{table}

The dynamics of the system for moderate values of the volume $V$ is rather difficult to guess intuitively. From simulations it appears that there is an alternation of patterns where two non consecutive chemical species are extinct or nearly extinct and the other two are more or less abundant (so called ``1-3 rich" or ``2-4 rich" patterns), moreover within the same pattern (``1-3 rich" form example), the system switches between two different sub-patterns where the non-extinct species are one abundant the other rare (say in pattern ``1-3 rich" there are sub-patterns 1A3R, where 1 is abundant, 3 rare and 1R3A). Transitions between sub-patterns (from 1A3R to 1R3A, for example) are much more frequent then transitions between patterns (from ``1-3 rich" to ``2-4 rich" and vice versa), but both kinds of transitions are driven by the event that one or a few molecules of the formerly extinct species enters the cycle and triggers with high probability a very fast cascade of events that is described in the original papers and that leads to a transition of one of the 2 possible kinds. For such transitions the authors coined the name DIT (Discretely Induced Transitions) since the single discrete event of the flow from the reservoirs of a molecule of a formerly extinct species may be responsible of a macroscopic event like a transition between different dynamic patterns or sub-patterns. Such transitions would be impossible in a system described by ODEs or by a diffusion process. To investigate the presence of such a switching behavior at different values of the volume $V$, in \cite{TogashiPRL} the authors suggested to consider the distribution at a large fixed time $t$ of the random variable
\begin{equation}U^{[V]}(t)=X^{[V]}_1(t) + X^{[V]}_3(t) - \left(X^{[V]}_2(t)+X^{[V]}_4(t)\right).\label{UU}\end{equation}
 While in ``1-3 rich" patterns, $U^{[V]}(t)$ is positive, in ``2-4 rich" patterns it assumes negative values. When $V<1/D$ the system keeps alternating the two patterns and the distribution of $U^{[V]}(t)$ is bimodal. For larger values of $V$ ($V=1/D$ for example) such distribution flattens and the bimodality gets replaced by a large plateaux. For even larger $V$ the noise starts not to be sufficient any more for the boundaries to be reached and than the distribution becomes unimodal with a peak at zero that sharpens with growing $V$ (the deterministic limit for $V\rightarrow\infty$ holds and in such regimen the classical approximation with continuous SDEs starts to work properly). 

 We aim at showing that it is actually possible to recover the behavior of the system even for moderate $V$ using the approximating process $Z^{[V]}(t)$ described in Section \ref{robabella} that is a continuous SDEs (thus a ``fluid" approximation) in the interior of the state space, but has positive occupation times and jumps at the boundaries where one or more chemical species are extinct. In the approximating process, ``discreteness" is kept only through the behavior at the boundaries, but the dynamics of the network is correctly reproduced.
 
 In Fig \ref{fig1} we display a simulated trajectory for each of the four components of the approximating process  $Z^{[V]}(t)$. The parameters are set to $V=32$, $D=1/256$, $Z^{[V]}(0)=(3,0,1,0)$. Around time $t=410$ we notice a switch between a ``1-3 rich" pattern to a ``2-4 rich" pattern. Switches between subpatterns are also easily detectable.
 
 In Figure \ref{fig2}, instead we plot a kernel density estimation of the variables 
 \begin{equation}\tilde U^{[V]}(t)=Z^{[V]}_1(t) + Z^{[V]}_3(t) - \left(Z^{[V]}_2(t)+Z^{[V]}_4(t)\right)\label{Utilde}\end{equation}
 derived from the approximated process, compared with the original counterpart $U^{[V]}(t)$ defined in equation \eqref{UU} simulated with the standard Gillespie simulation algorithm. The parameters are set to $V=128$, $D=1/256$, $Z^{[V]}(0)=(1,1,1,1)$ and the time at which positions are recorded is $t=2500$ which seems large enough not to feel the influence of the initial condition any more.  The agreement is very good.
 
 \begin{figure}[tbp]
 	\centering
 	\includegraphics[width=\textwidth]{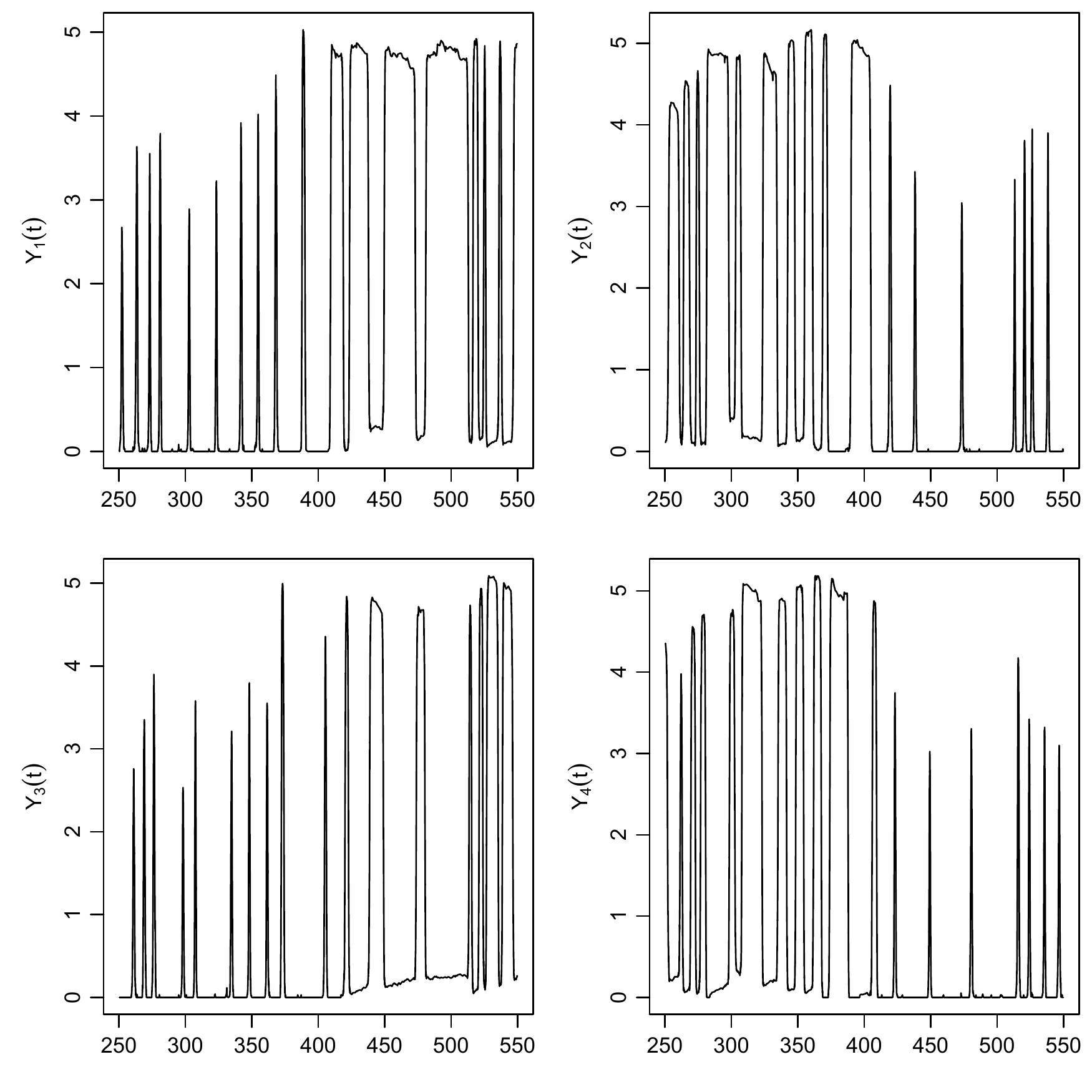}
 	\caption{A trajectory for each component of the approximating process $Z^{[32]}(t)$. Parameters are $V=32$, $D=1/256$ and $Z^{[32]}(t)= (3,0,1,0)$.\label{fig1}}
 \end{figure}
 
 \begin{figure}[tbp]
 	\centering
 	\includegraphics[width=.8\textwidth]{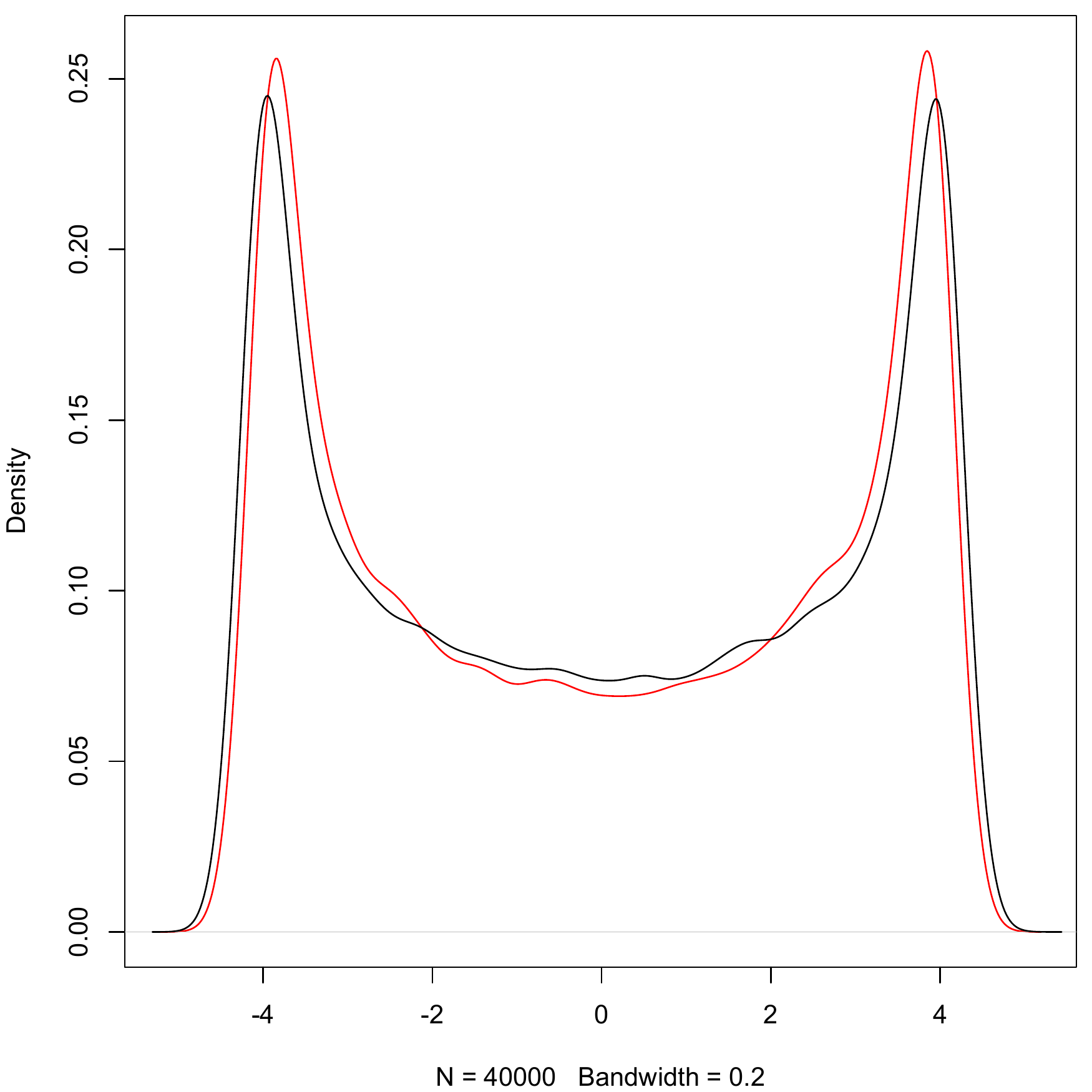}
 	\caption{Non parametric density estimation of the variables $U^{[V]}(t)$ (in red, derived from the original density process, cf. formula \eqref{UU}) and of $\tilde U^{[V]}(t)$  (in black, derived from the approximating jump-diffusion process, cf. formula \eqref{Utilde}) from a sample of 40.000 simulated paths. Parameters are set to $V=128$, $D=1/256$, $t=2500$, $Z^{[V]}(0)=(1,1,1,1)$.\label{fig2}}
 \end{figure}
 
\subsection{A further example and a comparison with Constrained Langevin Equations approach}\label{ruth}
In this last Session we replicate with our machinery an example that was proposed in \cite{ruth2017constrained} (Example 3) and studied in terms of the so-called Constrained Langevin equations. We refer to the preprint which is currently available at Ruth Williams web site. The reactions are:

\begin{align}
S_1 &\mathrel{\mathop{\rightleftarrows}^{{c_1}}_{{c_2}}} \emptyset,& S_2 &\mathrel{\mathop{\rightleftarrows}^{{c_3}}_{{c_4}}} \emptyset,& S_3 &\mathrel{\mathop{\rightleftarrows}^{{c_5}}_{{c_6}}} \emptyset,\\
S_3 +S_2&\mathrel{\mathop{\rightarrow}^{{c_7}}} 2 S_1,& 2S_1&\mathrel{\mathop{\rightarrow}^{{c_8}}} S_1+S_2,&S_1 +S_2&\mathrel{\mathop{\rightarrow}^{{c_9}}}  S_2
\end{align}

The vector of the reaction rate constants is set to 
\[c=\left(\frac{1}{\sqrt{10}}, 0.01,1,0.01,1,10,\frac{4}{5},1,\frac{3}{2\sqrt{10}} \right),\]
	while $V=100$. 
	
	The corresponding deterministic models is bistable with the following stable equilibria
	\[\begin{aligned}
	x_1&=(0.12679,2.90328\cdot 10^{-3},9.97683)\\
	x_2&=(2.96686,2.31681,3.50454)
	\end{aligned}\]
	
	In the stochastic system trajectories which are initialized at $x_0=(0.1,0.1,10)$, thus in the vicinity of $x_1$, tend to wander around such equilibrium point until the noise brings them into the domain of attraction of $x_2$ around which they spend a much longer time. We stop the simulations at $T=100$, the majority of the trajectories are already around $x_2$, but a considerable number is still close to $x_1$ (nearly the 40\%). The time step for the Euler discretization for the jump-diffusion is fixed to $0.005$ when each component of the process is either larger then $6/V$ (thus sufficiently far away from 0) or at 0. If there are components of the process in the region between $3/V$ and $6/V$ we reduce the step to $0.0007$, and if any component is in the region between $10^-14$ and $3/V$ we reduce it further to $0.0002$. Such a reduction is needed not to miss hidden crossing of the boundary between to discretized observations that would have a strong impact on the quality of the approximation.
	A trajectory of the density process $X^{[V]}(t)$ and an independent one of the approximating jump-diffusion $Z^{[V]}(t)$ are plotted in Figure \ref{fig3}. Let us remark that the two trajectories are not coupled so to satisfy Theorem \ref{th:main}, but still they both display the qualitative behavior described above.
	
	\begin{figure}[tbp]
		\centering
		\includegraphics[width=1\textwidth]{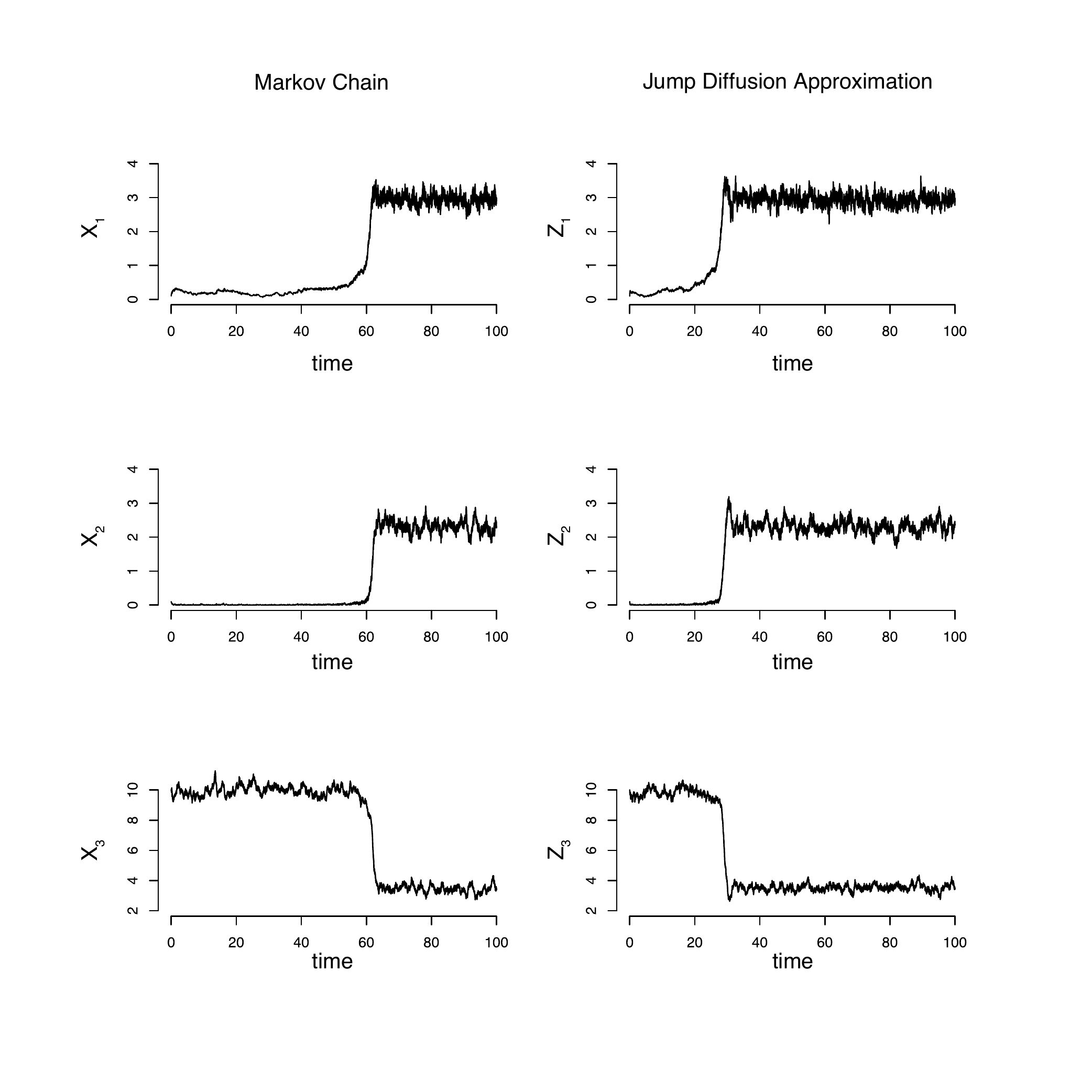}
		\caption{A trajectory of the 3 components of the density process $X^{[V]}(t)$ and an independent one of the jump-diffusion approximation. The two are not coupled so to satisfy Theorem \ref{th:main}, but they show the same qualitative behavior described in the text. \label{fig3}}.
	\end{figure}
	
	In Figure \ref{fig4}, we plot a \emph{heat} scatter plot of the collection of points obtained by sampling the processes every time interval of size $0.2$ over 400 repetitions. Both the density process  $X^{[V]}(t)$, the jump diffusion $Z^{[V]}(t)$ and the Constrained Langevin Equation are reported. The Constrained Langevin Equation is simulated with an Euler scheme with fixed step 0.005 with the same code that was used in \cite{ruth2017constrained} and described there in more details. We thank the authors for having agreed to share their code.  Again the agreement is very good. 

\begin{figure}[tbp]
	\centering
	\includegraphics[width=\textwidth]{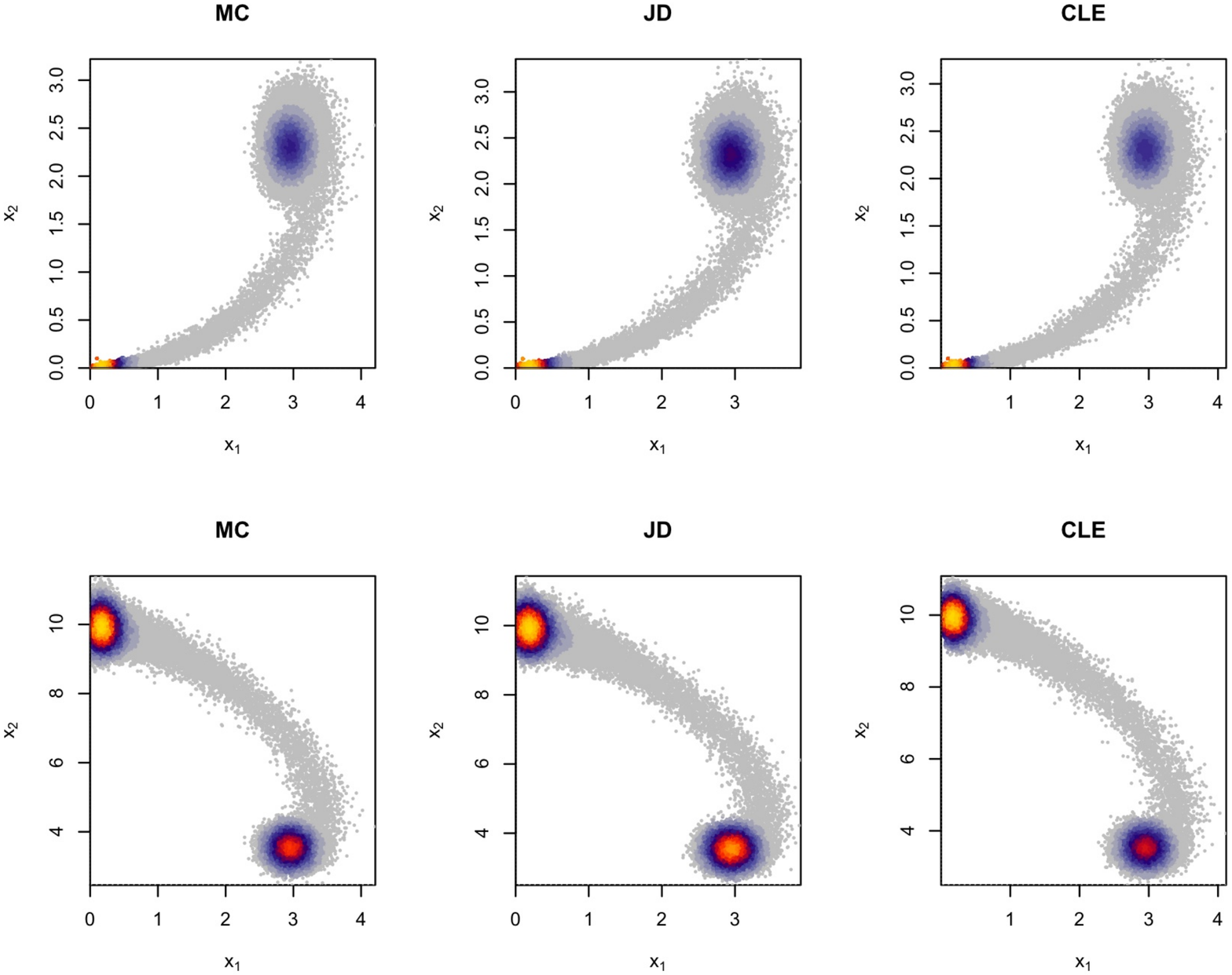}
	\caption{\emph{Heat} scatterplot of points generated from 400 trajectories of the density process (MC), the jump-diffusion approximation (JD) and the Constrained Langevin equations approximation (CLE) sampled every time interval of length 0.2.\label{fig4}}.
\end{figure}

\section{Acknowledgements}
We thank Tom Kurtz for having found a relevant mistake in the first draft of this paper and for having  given some suggestions on how to correct it. We also thank Cristina Costantini for having been very kind and helpful in different occasions.
We thank Ruth Williams and Saul Leite for having agreed to share their code for simulating the example in Section \ref{ruth} with the Constrained Langevin Equations approach, and for producing Figure \ref{fig4}.

\bibliographystyle{apalike}
\bibliography{Bibl}

\end{document}